\documentclass[12pt]{article}
\usepackage{geometry,amsmath,amssymb,amsthm,graphicx,epsfig}
\usepackage{enumerate}

\usepackage[colorlinks=true,linkcolor=blue,citecolor=blue]{hyperref}

\theoremstyle{plain}
  \newtheorem{theorem}{Theorem}
  
  \newtheorem{conjecture}{Conjecture}
  \newtheorem{corollary}{Corollary}
  \newtheorem{proposition}{Proposition}
\theoremstyle{definition}
  \newtheorem{example}{Example}
  \newtheorem{remark}{Remark}

\newcommand{\qede}{\hspace*{\fill}$\Diamond$\medskip}
\newcommand{\stack}[2]{\genfrac{}{}{0pt}{}{#1}{#2}}

\newcommand{\op}[1]{\ensuremath{\operatorname{#1}}}
\renewcommand{\Re}{\op{Re}\,}
\renewcommand{\Im}{\op{Im}\,}
\newcommand{\e}{\mathrm{e}}
\newcommand{\G}{\mathrm{G}}

\newcommand{\md}{\mathrm{d}}
\newcommand{\id}{\,\mathrm{d}}

\newcommand{\pFq}[5]{\ensuremath{{}_{#1}F_{#2} \left( \genfrac{}{}{0pt}{}{#3}{#4} \bigg| {#5} \right)}}
\newcommand{\Li}{\op{Li}}

\newcommand{\Ti}[2]{\op{Ti}_{#1}\left( {#2} \right)}
\newcommand{\Gl}[2]{\op{Gl}_{#1}\left( {#2} \right)}
\newcommand{\Cl}[2]{\op{Cl}_{#1}\left( {#2} \right)}
\newcommand{\mgl}[1]{\Gl{#1}{\frac{\pi}{3}}}
\newcommand{\mcl}[1]{\Cl{#1}{\frac{\pi}{3}}}

\newcommand{\Ls}[2]{\op{Ls}_{#1}\left( {#2} \right)}
\newcommand{\LsD}[3]{\op{Ls}_{#1}^{(#2)}\left( {#3} \right)}
\newcommand{\Lsc}[3]{\op{Lsc}_{#1,#2}\left( {#3} \right)}

\title{Log-sine evaluations of Mahler measures, II\\
{\small This paper is dedicated to the memory of John Selfridge}}

\author{David Borwein\thanks{Dept of Mathematics, University of Western Ontario, London, ONT,  N6A 5B7, Canada. Email: \texttt{dborwein@uwo.ca}}
,~ Jonathan M. Borwein\thanks{Centre for Computer-assisted Research Mathematics and its Applications (CARMA), School of Mathematical and Physical Sciences, University of Newcastle, Callaghan, NSW 2308,
Australia. Email:   \texttt{jonathan.borwein@newcastle.edu.au},
\texttt{jborwein@cs.dal.ca}. Supported in part by the Australian
Research Council and the University of Newcastle.},
Armin Straub\thanks{Tulane University, Louisiana, USA. Email: \texttt{astraub@tulane.edu}}
~and James Wan\thanks{CARMA, University of Newcastle. Email: \texttt{james.wan@newcastle.edu.au}}
}

\date{\ttfamily \today}

\begin{document}
\maketitle

\begin{abstract}
  We continue our analysis of higher and multiple Mahler measures using
  log-sine integrals as started in \cite{logsin1, logsin-evaluations}.  This
  motivates a detailed study of various multiple polylogarithms \cite{b3l} and worked examples are given.
 Our techniques enable the reduction of several multiple Mahler measures, and
 supply an easy proof of two conjectures by Boyd.
\end{abstract}

\section{Introduction}\label{sec:intro}

In \cite{logsin1} the classical log-sine integrals and their extensions were
used to develop a variety of results relating to higher and multiple Mahler
measures \cite{boyd,klo}. The utility of this approach was such that we
continue the work herein. Among other things, it allows us to tap into a rich
analytic literature \cite{lewin2}. In \cite{logsin-evaluations} the
computational underpinnings of such studies are illuminated. The use of related
integrals is currently being exploited for multi-zeta value studies
\cite{ona11}.  Such evaluations are also useful for physics \cite{lsjk}:
log-sine integrals appeared for instance  in the calculation of
higher terms in the $\varepsilon$-expansion of various Feynman diagrams
\cite{davydychev-eps, kalmykov-eps}. Of particular importance are the log-sine
integrals at the special values $\pi/3$, $\pi/2$, $2\pi/3$, and $\pi$. The
log-sine integrals also appear in many settings in number theory and analysis:
classes of inverse binomial sums can be expressed in terms of generalized
log-sine integrals \cite{mcv, davydychev-bin}.

The structure of this article is as follows. In Section \ref{sec:prelim} our
basic tools are described. After providing necessary results on log-sine
integrals in Section \ref{sec:logsin}, we turn to relationships between random
walks and Mahler measures in Section \ref{sec:walks}. In particular, we will be
interested in the multiple Mahler measure $\mu_n(1+x+y)$ which has a fine
hypergeometric form \eqref{eq:hyper} and a natural trigonometric
representation \eqref{eq:mu-trigintegral} as a double integral.

In Section \ref{sec:epsilon} we directly expand \eqref{eq:hyper} and use
known results from the $\varepsilon$-expansion of hypergeometric functions
\cite{dk1-eps, davydychev-bin} to obtain $\mu_n(1+x+y)$ in terms of multiple
inverse binomial sums.  In the cases $n=1,2,3$ this leads to explicit
polylogarithmic evaluations.

An alternative approach based of the double integral representation
\eqref{eq:mu-trigintegral} is taken up in Section \ref{sec:mun} which considers
the evaluation of the inner integral in \eqref{eq:mu-trigintegral}.  Aided by
combinatorics, we show in Theorems \ref{thm:mun} and \ref{thm:rhon} that these
can always be expressed in terms of multiple harmonic polylogarithms of weight
$k$. Accordingly, we show in Section \ref{sec:lireductions} how these
polylogarithms can be reduced explicitly for low weights. In Section
\ref{sec:mu2} we reprise from \cite{logsin1} the evaluation of $\mu_2(1+x+y)$.
Then in Section \ref{sec:mu3} we apply our general results from Section
\ref{sec:mun} to a conjectural evaluation of $\mu_3(1+x+y)$.

In Section \ref{sec:boyd} we finish with an elementary proof of two recently
established 1998 conjectures of Boyd and use these tools to obtain a new Mahler
measure.

\section{Preliminaries}\label{sec:prelim}

For $k$ functions (typically Laurent polynomials) in $n$ variables the
\emph{multiple Mahler measure}, introduced in \cite{klo}, is defined by
\begin{equation*}
  \mu(P_1,P_2, \ldots, P_k)
  := \int_0^1 \cdots \int_0^1 \prod_{j=1}^k \log \left| P_j\left(e^{2\pi i t_1},
    \ldots, e^{2\pi i t_n}\right)\right| \md t_1 \md t_2 \ldots \md t_n.
\end{equation*}
When $P=P_1=P_2= \cdots =P_k$ this devolves to a \emph{higher Mahler measure},
$\mu_k(P)$, as introduced and examined in \cite{klo}. When $k=1$ both reduce
to the standard (logarithmic) \emph{Mahler measure} \cite{boyd}.

We also recall \emph{Jensen's formula}:
\begin{equation}\label{eq:sl}
  \int_0^1 \log \left|\alpha-e^{2\pi i\,t}\right|\id t
  = \log \left(\max\{|\alpha|,1\}\right).
\end{equation}

An easy consequence of Jensen's formula is that for complex constants $a$ and
$b$ we have
\begin{equation}\label{eq:mlin}
  \mu(ax+b) = \log|a|\vee\log |b|.
\end{equation}

In the following development,
\begin{equation*}
  \Li_{a_1,\ldots,a_k}(z)
  := \sum_{n_1>\cdots>n_k>0} \frac{z^{n_1}}{n_1^{a_1} \cdots n_k^{a_k}}
\end{equation*}
denotes the \emph{generalized polylogarithm} as is studied in \cite{mcv} and in
\cite[Ch. 3]{bbg}.  For example, $\Li_{2,1}(z) =\sum_{k=1}^\infty
\frac{z^k}{k^2}\sum_{j=1}^{k-1} \frac1j$.  In particular, $\Li_k(x) :=
\sum_{n=1}^\infty \frac{x^n}{n^k}$ is the \emph{polylogarithm of order $k$} and
\begin{equation*}
  \Ti{k}{x} := \sum_{n=0}^\infty \,(-1)^n\frac{ x^{2n+1}}{(2n+1)^{k}}
\end{equation*}
is the related \emph{inverse tangent of order $k$}. We use the same notation for
the analytic continuations of these functions.

Moreover, \emph{multiple zeta values} are denoted by
\begin{equation*}
  \zeta(a_1,\ldots,a_k) := \Li_{a_1,\ldots,a_k}(1).
\end{equation*}
Similarly, we consider the \emph{multiple Clausen functions} ($\op{Cl}$) and
\emph{multiple Glaisher functions} ($\op{Gl}$) of depth $k$ defined as
\begin{align}
  \Cl{a_1,\ldots,a_k}{\theta} &= \left\{ \begin{array}{ll}
    \Im \Li_{a_1,\ldots,a_k}(e^{i\theta}) & \text{if $w$ even}\\
    \Re \Li_{a_1,\ldots,a_k}(e^{i\theta}) & \text{if $w$ odd}
  \end{array} \right\}, \label{eq:defcl}\\
  \Gl{a_1,\ldots,a_k}{\theta} &= \left\{ \begin{array}{ll}
    \Re \Li_{a_1,\ldots,a_k}(e^{i\theta}) & \text{if $w$ even}\\
    \Im \Li_{a_1,\ldots,a_k}(e^{i\theta}) & \text{if $w$ odd}
  \end{array} \right\}, \label{eq:defgl}
\end{align}
where $w=a_1+\ldots+a_k$ is the weight of the function.
As illustrated in \eqref{eq:defcl2}, the Clausen and Glaisher functions
alternate between being cosine and sine series with the parity of the
dimension. Of particular importance will be the case of $\theta = \pi/3$ which
has also been considered in \cite{mcv}.

Our other notation and usage is largely consistent with that in \cite{lewin2}
and the newly published \cite{NIST}, in which most of the requisite
material is described.  Finally, a recent elaboration of what is meant when we
speak about evaluations and ``closed forms'' is to be found in \cite{closed}.

\section{Log-sine integrals} \label{sec:logsin}

For $n=1,2, \ldots$, we consider the \emph{log-sine integrals} defined by
\begin{equation}\label{eq:slx}
  \Ls{n}{\sigma} := -\int_{0}^{\sigma}\log^{n-1} \left|2\,\sin \frac\theta 2\right|\,{\md\theta}
\end{equation}
and their moments for $k\ge 0$ given by
\begin{equation}\label{eq:slxm}
  \LsD{n}{k}{\sigma}
  := -\int_{0}^{\sigma}\theta^k\,\log^{n-1-k} \left|2\,\sin \frac\theta 2\right| \,{\md\theta}.
\end{equation}
In each case the modulus is not needed for $0 \le \sigma \le 2\pi$.  Various
log-sine integral evaluations may be found in \cite[\S7.6 \& \S7.9]{lewin2}.

We observe that $\Ls{1}{\sigma} = -\sigma$ and that $\LsD{n}{0}{\sigma} =
\Ls{n}{\sigma}$.  This is the notation used by Lewin \cite{lewin, lewin2}. In
particular,
\begin{equation}\label{eq:defcl2}
  \Ls{2}{\sigma} = \Cl{2}{\sigma}
  := \sum_{n=1}^\infty \frac{\sin(n\sigma)}{n^2}
\end{equation}
is the Clausen function introduced in \eqref{eq:defcl}.

\subsection[Log-sine integrals at pi]{Log-sine integrals at $\pi$}\label{ssec:lspi}

We first recall that the log-sine integrals at $\pi$ can always be evaluated in
terms of zeta values.  This is a consequence of the exponential generating
function \cite[Eqn.  (7.109)]{lewin2}
\begin{equation}\label{eq:lsatpiogf}
  -\frac1\pi\,\sum_{m=0}^{\infty} \Ls{m+1}{\pi} \,\frac{u^m}{m!}
  =\frac{\Gamma\left(1+u\right)}{\Gamma^2\left(1+\frac u2\right)}
  = \binom{u}{u/2}.
\end{equation}
This will be revisited and further explained in Section \ref{sec:walks}.

\begin{example}[Values of $\Ls{n}{\pi}$]\label{ex:pi}
  For instance, we have $\Ls{2}{\pi} = 0$ as well as
  \begin{align*}
    -\Ls{3}{\pi} &= {\frac{1}{12}}\,\pi^3\\
    \Ls{4}{\pi} &= \frac32\pi\,\zeta(3)\\
    -\Ls{5}{\pi} &= \frac{19}{240}\,\pi^5\\ 
    \Ls{6}{\pi} &= {\frac{45}{2}}\,\pi \,\zeta(5)+\frac54\,\pi^3\zeta(3)\\
   - \Ls{7}{\pi} &= {\frac{275}{1344}}\,\pi^7+{\frac{45}{2}}\,\pi \, \zeta^2(3)\\
    \Ls{8}{\pi} &= {\frac{2835}{4}}\,\pi \,\zeta(7)+ {\frac{315}{8}}\,\pi^3\zeta(5)
      +{\frac{133}{32}}\,\pi^5\zeta(3),
  \end{align*}
  and so forth.  Note that these values may be conveniently obtained from
  \eqref{eq:lsatpiogf} by a computer algebra system as the following snippet of
  \emph{Maple} code demonstrates:\\
  \texttt{for k to 6 do simplify(subs(x=0,diff(Pi*binomial(x,x/2),x\$k))) od}\\
  More general log-sine evaluations with an emphasis on automatic evaluations
  have been studied in \cite{logsin-evaluations}.
  \qede
\end{example}

For general log-sine integrals, in \cite{logsin-evaluations} the following
computationally effective exponential generating function was obtained.

\begin{theorem}[Generating function for $\LsD{n+k+1}{k}{\pi}$]\label{prop:gfb}
For $2|\mu| <\lambda <1$ we have
  \begin{align}
    \sum_{n,k \ge 0}\LsD{n+k+1}{k}{\pi}\frac{\lambda^n}{n!}\frac{(i\mu)^k}{k!}
        &= -i \sum_{n\ge0} \binom{\lambda}{n}
    \frac{(-1)^n\, \e^{i\pi\frac\lambda2} - \e^{i\pi\mu}}{\mu-\frac\lambda2+n}. \label{eq:gfsum}
  \end{align}
\end{theorem}

One may extract one-variable generating functions from \eqref{eq:gfsum}. For instance,
\begin{equation*}
  \sum_{n=0}^\infty \LsD{n+2}{1}{\pi}\frac{\lambda^n}{n!}
  = \sum_{n=0}^{\infty} \binom{\lambda}{n} \frac{-1 + (-1)^n \cos\frac{\pi\lambda}2}{\left(n -\frac\lambda2 \right)^2}.
\end{equation*}
The log-sine integrals at $\pi/3$  are especially useful as illustrated in
\cite{mcv} and are discussed at some length in \cite{logsin1} where other
applications to Mahler measures are given.

\subsection{Extensions of the log-sine integrals}\label{ssec:ls3}

It is possible to extend some of these considerations to the log-sine-cosine integrals
\begin{equation}\label{eq:deflsc}
  \Lsc{m}{n}{\sigma} := -\int_{0}^{\sigma}\log^{m-1} \left|2\,\sin \frac \theta 2\right|\,\log^{n-1} \left|2\,\cos \frac \theta 2\right| \id\theta.
\end{equation}
Then $\Lsc{m}{1}{\sigma}= \Ls{m}{\sigma}$.
Let us set
\begin{equation}
  \mu_{m,n}(1-x,1+x) :=\mu(\underbrace{1-x, \cdots,1-x}_m,\underbrace{1+x, \cdots,1+x}_n).
\end{equation}

Then immediately from the definition we obtain the following:

\begin{theorem}[Evaluation of $\mu_{m,n}(1-x,1+x)$]\label{thm:gls}
  For non-negative integers $m,n$,
  \begin{equation}\label{eq:gslmahler}
    \mu_{m,n}(1-x,1+x) = -\frac1\pi\,\Lsc{m+1}{n+1}{\pi}.
  \end{equation}
\end{theorem}

In every case this is evaluable in terms of zeta values. Indeed, using the result in \cite[\S7.9.2, (7.114)]{lewin2}, we obtain the
generating function
\begin{equation}\label{eq:glsatpiogf}
  {\rm gs}(u,v):=-\frac1\pi\, \sum_{m,n=0}^\infty \Lsc{m+1}{n+1}{\pi} \,\frac{u^m}{m!}\frac{v^n}{n!}
  =\frac{2^{u+v}}\pi\,\frac{\Gamma \left( \frac{1+u}2 \right)\Gamma \left(\frac{1+v}2 \right)}{\Gamma \left( 1+\frac{u+v}2 \right)}.
\end{equation}
From the duplication formula for the Gamma function this can be rewritten as
\[{\rm gs}(u,v) = \binom{u}{\frac u2} \binom{v}{\frac v2} {\frac{\Gamma \left(
1+\frac u2 \right) \Gamma \left(1+\frac v2 \right) }{\Gamma \left( 1+\frac{u+v}2 \right) }},
\]
so that
\[{\rm gs}(u,0) = \binom{u}{\frac u2} = {\rm gs}(u,u).\]
From here it is apparent that \eqref{eq:glsatpiogf} is an extension of
\eqref{eq:lsatpiogf}:

\begin{example}[Values of $\Lsc{n}{m}{\pi}$]\label{ex:gpi}
  For instance,
  \begin{align*}\label{ex:gls}
    \mu_{2,1}(1-x,1+x) &= \mu_{1,2}(1-x,1+x) = \frac{1}{4}\zeta(3), \\
    \mu_{3,2}(1-x,1+x) &= \frac{3}{4}\zeta(5)-\frac{1}{8}\pi^2\zeta(3), \\
    \mu_{6,3}(1-x,1+x) &=
    \frac{315}{4}\zeta(9) + \frac{135}{32}\pi^2\zeta(7) + \frac{309}{128}\pi^4\zeta(5)
    - \frac{45}{256}\pi^6\zeta(3) - \frac{1575}{32}\zeta^3(3).
  \end{align*}
  As in Example \ref{ex:pi} this can be easily obtained with a line of code
  in a computer algebra system such as \emph{Mathematica} or \emph{Maple}.
  \qede
\end{example}

\begin{remark}
  Also ${\rm gs}(u,-u)=\sec(\pi/(2u)).$ From this we deduce for $n=0,1,2,\ldots$ that
  \[\sum_{k=0}^n (-1)^k \mu_{k,n-k}(1-x,1+x) = |E_{2n}|\,\frac{\left(\frac\pi2\right)^{2n}}{(2n)!} =\frac4\pi \,L_{-4}(2n+1),\]
  where $E_{2n}$ are the even Euler numbers: $1,-1,5,-61,1385 \ldots$\,.
  \qede
\end{remark}

A more recondite \emph{extended log-sine integral of order three} is developed
in \cite[\S 8.4.3]{lewin2} from properties of the trilogarithm. It is defined by
\begin{align}
  \Ls3{\theta,\omega}
  &:=-\int_0^\theta\,\log \left|2\sin \frac\sigma 2\right|\, \log\left|2\sin \frac{ \sigma+\omega}2 \right| \id \sigma,
\end{align}
so that $\Ls3{\theta,0} = \Ls3{\theta}.$
This extended log-sine integral reduces as follows:
\begin{align}\label{eq:ls3}
  -\Ls3{2\theta, 2\omega}
  &= \frac12 \Ls3{2\omega} -\frac12 \Ls3{2\theta}
  -\frac12 \Ls3{2\theta+2\omega} \nonumber\\
  &\quad- 2\Im\Li_3\left(\frac{\sin(\theta) \e^{i\omega}}{\sin(\theta+\omega)}\right)
  +\theta\log^2 \left( \frac{\sin(\theta)}{\sin(\theta+\omega)}\right) \nonumber \\
  &\quad+ \log\left( {\frac{\sin(\theta)}{\sin(\theta+\omega)}} \right)
  \left\{ \Cl{2}{2\theta} + \Cl{2}{2\omega} - \Cl{2}{2\theta+2\omega} \right\}.
\end{align}

We note that $-\frac{1}{2\pi} \Ls3{2\pi,\omega} = \mu(1-x,1-e^{i\omega}x)$ but
this is more easily evaluated by Fourier techniques. Indeed one has:

\begin{proposition}[A dilogarithmic measure, part I \cite{klo}]\label{prop:diloga}
  For two complex numbers $u$ and $v$ we have
  \begin{equation}\mu(1-u\,x,1-v\,x) =\left\{
    \begin{array}{ll}
      \frac 12\,\Re\Li_2\left(v\overline{u} \right), & \text{if $|u|  \le 1,|v| \le 1$,} \\
      \frac 12\,\Re\Li_2\left(\frac v{\overline{u}}\right), & \text{if $|u| \ge 1,|v| \le 1$,} \\
      \frac 12\,\Re\Li_2\left(\frac 1{v\overline{u}} \right)+ \log |u| \log |v| , & \text{if $|u|\ge 1 ,|v| \ge 1$.}
    \end{array}
    \right.
  \end{equation}
\end{proposition}

This is proven much as is \eqref{eq:mu2} of Proposition \ref{prop:mu2}. In
Lewin's terms \cite[A.2.5]{lewin2} for $0 < \theta\le 2\pi$ and $r\ge0$ we may
write
\begin{equation}\label{def:Li}
  \Re\Li_2\left(re^{i\theta}\right) =: \Li_2\left(r,\theta\right)= -\frac 12\int_0^r \log\left(t^2+1-2t \cos \theta\right)\frac{\md t}{t},
\end{equation}
with the reflection formula
\begin{equation}
  \Li_2\left(r,\theta\right) + \Li_2\left(\frac1r,\theta\right)
  = \zeta(2)- \frac12\,\log^2 r+\frac12\,(\pi-\theta)^2.
\end{equation}
This leads to:

\begin{proposition}[A dilogarithmic measure, part II]\label{prop:dilogb}
  For complex numbers $u=re^{i\theta}$ and $v=se^{i\tau}$ we have
  \begin{equation}\mu(1-u \, x,1-v \, x) =\left\{
    \begin{array}{ll}
      \frac 12\,\Li_2\left(rs,\theta-\tau\right)  & \text{if $r \le 1,s \le 1$,} \\
      \frac 12\,\Li_2\left(\frac s r,\theta+\tau\right), & \text{if $r\ge 1,s \le 1$,} \\
      \frac 12\,\Li_2\left(\frac1{sr},\theta-\tau\right)+ \log r \log s , & \text{if $r\ge 1 ,s \ge 1$.}
    \end{array}
    \right.
  \end{equation}
\end{proposition}

\section{Mahler measures and moments of random walks}\label{sec:walks}

The $s$-th moments of an $n$-step uniform random walk are given by
\[ W_n (s) = \int_0^1 \ldots \int_0^1 \left| \sum_{k = 1}^n e^{2 \pi
   i t_k} \right|^s \md t_1 \cdots \md t_n \]
and their relation with Mahler measure is observed in \cite{bswz-densities}.
In particular,
\[ W_n' (0)  = \mu (1 + x_1 + \ldots + x_{n-1}), \]
with the cases $2 \le n \le 6$
discussed in \cite{logsin1}.

Higher derivatives of $W_n$ correspond to higher Mahler measures:
\begin{equation}\label{eq:diffW}
  W_n^{(m)}(0) = \mu_m (1 + x_1 + \ldots + x_{n-1}).
\end{equation}
The evaluation $W_2(s) = \binom{s}{s/2}$  thus
explains and proves the generating function \eqref{eq:lsatpiogf}; in other
words, we find that
\begin{align}\label{w3dm}
  W_2^{(m)}(0) &= -\frac1\pi \Ls{m+1}{\pi}.
\end{align}

As a consequence of the study of random walks in \cite{bswz-densities} we
record the following generating function for $\mu_k(1+x+y)$ which follows from
\eqref{eq:diffW} and the hypergeometric expression for $W_3$ in \cite[Thm.
10]{bswz-densities}. There is a corresponding expression, using a single
Meijer-$G$ function, for $W_4$ (i.e., $\mu_m(1+x+y+z)$) given in \cite[Thm.
11]{bswz-densities}.

\begin{theorem}[Hypergeometric form for $W_3(s)$]\label{thm:W3}
  For complex $|s|<2$, we may write
  \begin{align}\label{eq:hyper}
    W_3(s) = \sum_{n=0}^\infty \mu_n(1+x+y)\frac{s^n}{n!}
    &=\frac {\sqrt{3}}{2\pi}\, 3^{s+1}
    \,\frac{\Gamma(1+\frac s2)^2}{\Gamma(s+2)}\,
    \pFq32{\frac{s+2}2, \frac{s+2}2,\frac{s+2}2}{1,\frac{s+3}2}{\frac14} \\
    &=\frac{\sqrt {3}}{\pi }\left(\frac 32\right)^{s+1}\int _{0}^{1}\!{\frac {{z}^{1+s}
    {\pFq21{1+\frac s2,1+\frac s2}{1}{\frac{z^2}4}}}{\sqrt
    {1-{z}^{2}}}}\,{\id z}. \label{eq:hyperb}
  \end{align}
\end{theorem}

\begin{proof}
  Equation  \eqref{eq:hyper} is proven in \cite[Thm. 10]{bswz-densities}, while
  \eqref{eq:hyperb} is a consequence of \eqref{eq:hyper} and \cite[Eqn.
  (16.5.2)]{NIST}.
\end{proof}

We shall exploit Theorem \ref{thm:W3} next, in Section \ref{sec:epsilon}. For
integers $n \ge 1$ we also have
\begin{equation}\label{eq:mu-trigintegral}
  \mu_n(1+x+y) =\frac{1}{4\pi^2} \int_0^{2\pi} \md\theta
    \int_0^{2\pi} \left(\Re\log\left(1-2\sin(\theta)\e^{i\,\omega}\right)\right)^n\id \omega,
\end{equation}
as follows directly from the definition and some simple trigonometry, since
$\Re\log =\log|\cdot|$. This is the basis for the  evaluations of Section
\ref{sec:mun}. In particular, in Section \ref{sec:mun} we will evaluate the
inner integral in terms of multiple harmonic polylogarithms.

\section{Epsilon expansion of $W_3$}\label{sec:epsilon}

In this section we use known results from the $\varepsilon$-expansion of
hypergeometric functions \cite{dk1-eps, davydychev-bin} to obtain
$\mu_n(1+x+y)$ in terms of multiple inverse binomial sums. We then derive
complete evaluations of $\mu_1(1+x+y)$, $\mu_2(1+x+y)$ and $\mu_3(1+x+y)$.  An
alternative approach will be pursued in Sections \ref{sec:mun} and
\ref{sec:mu23}.

In light of Theorem \ref{thm:W3}, the evaluation of $\mu_n(1+x+y)$
is essentially reduced to the Taylor expansion
\begin{align}\label{eq:alphadef}
   \pFq32{\frac{\varepsilon+2}2, \frac{\varepsilon+2}2,\frac{\varepsilon+2}2}{1,\frac{\varepsilon+3}2}{\frac14}
   = \sum_{n=0}^\infty \alpha_n \varepsilon^n.
\end{align}
Indeed, from \eqref{eq:hyper} and Leibniz' rule we have
\begin{align}\label{eq:mun-sum}
  \mu_n(1+x+y) &= \frac {\sqrt{3}}{2\pi} \sum _{k=0}^n \binom{n}{k} \alpha_k \beta_{n-k}
\end{align}
where $\beta_k$ is defined by
\begin{align}\label{eq:beta-def}
   3^{ \varepsilon+1} \,\frac{\Gamma(1+\frac \varepsilon2)^2}{\Gamma(2+\varepsilon)}
   = \sum_{n=0}^\infty \beta_n \varepsilon^n.
\end{align}
Note that $\beta_k$ is easy to compute as illustrated in Example \ref{ex:pi}.
% Alternatively, one may use
% \begin{align*}
%   \beta_n=\frac3{ n!}\int_0^{\frac \pi 2} \sin u\log^n\left(\frac32\sin u\right) \id u.
% \end{align*}
The expansion of hypergeometric functions in terms of their parameters as in
\eqref{eq:alphadef} occurs in physics \cite{dk1-eps, davydychev-bin} in the
context of the evaluation of Feynman diagrams and is commonly referred to
as \emph{epsilon expansion}, thus explaining the choice of variable in
\eqref{eq:alphadef}.

\begin{remark}
  From \eqref{eq:beta-def} we see that $\beta_k$ may be computed directly from
  the coefficients $\gamma_k$ defined by the Taylor expansion
  \begin{equation*}
    \frac{\Gamma(1+\frac\varepsilon2)^2}{\Gamma(1+\varepsilon)}
    = \frac{1}{\binom{\varepsilon}{\varepsilon/2}} = \sum_{n=0}^\infty \gamma_n \varepsilon^n.
  \end{equation*}
  Appealing to \eqref{eq:lsatpiogf} we find that $\gamma_n$ is recursively determined
  by $\gamma_0=1$ and
  \begin{equation*}
    \gamma_n = \frac{1}{\pi} \sum_{k=1}^n \Ls{k+1}{\pi} \frac{\gamma_{n-k}}{k!}.
  \end{equation*}
  In particular, the results of Section \ref{ssec:lspi} show that $\gamma_n$ can
  always be expressed in terms of zeta values. Accordingly, $\beta_n$ evaluates
  in terms of $\log3$ and zeta values.
  \qede
\end{remark}

Let $S_k(j) := \sum_{m=1}^j \frac{1}{m^k}$ denote the harmonic numbers of order
$k$.  Following \cite{davydychev-bin} we abbreviate $S_k:=S_k(j-1)$ and
$\bar{S}_k:=S_k(2j-1)$ in order to make it more clear which results in this
reference contribute to the evaluations below.  As in \cite[Appendix
B]{dk1-eps}, we use the duplication formula $(2a)_{2j} = 4^j (a)_j (a+1/2)_j$
as well as the expansion
\begin{align}
  % \frac{(1+a\varepsilon)_j}{j!} = \exp\left[ -\sum_{k=1}^\infty \frac{(-a\varepsilon)^k}{k} S_k(j) \right]
  \frac{(m+a\varepsilon)_j}{(m)_j} = \exp\left[ -\sum_{k=1}^\infty \frac{(-a\varepsilon)^k}{k} \left[S_k(j+m-1) - S_k(m-1)\right] \right]
\end{align}
where $m$ is a positive integer
to write:
\begin{align}
\nonumber  \pFq32{\frac{\varepsilon+2}2,
\frac{\varepsilon+2}2,\frac{\varepsilon+2}2}{1,\frac{\varepsilon+3}2}{\frac14}
  &= \sum_{j=0}^\infty \frac{(1+\varepsilon/2)_j^3}{4^j (j!)^2 (3/2+\varepsilon/2)_j}
  \\ \nonumber
  &= \sum_{j=0}^\infty \frac{(1+\varepsilon/2)_j^4}{(j!)^2 (2+\varepsilon)_{2j}}
  \\ \nonumber
  &= \sum_{j=0}^\infty \frac{2}{j+1} \frac{1}{\binom{2(j+1)}{j+1}}
  \left[\frac{(1+\varepsilon/2)_j}{j!}\right]^4 \left[\frac{(2+\varepsilon)_{2j}}{(2j+1)!}\right]^{-1} \\
  % &= \sum_{j=0}^\infty \frac{2}{j+1} \frac{1}{\binom{2(j+1)}{j+1}}
  % \exp\left[ \sum_{k=1}^\infty \frac{(-\varepsilon)^k}{k} \left[S_k(2j+1)-S_k(1) - 4\frac{S_k(j)}{2^k}\right] \right] \\
  % &= \sum_{j=1}^\infty \frac{2}{j} \frac{1}{\binom{2j}{j}}
  % \exp\left[ \sum_{k=1}^\infty \frac{(-\varepsilon)^k}{k} \left[S_k(2j-1)-1 - 4\frac{S_k(j-1)}{2^k}\right] \right] \\
  &= \sum_{j=1}^\infty \frac{2}{j} \frac{1}{\binom{2j}{j}}
  \exp\left[ \sum_{k=1}^\infty \frac{(-\varepsilon)^k}{k} A_{k,j} \right] \label{eq:hypeps}
\end{align}
where
\begin{align}\label{eq:akj}
  A_{k,j} :=S_k(2j-1)-1 -
  4\frac{S_k(j-1)}{2^k} =
  \sum_{m=2}^{2j-1}\frac{2(-1)^{m+1}-1}{m^k}.
\end{align}

We can now read off the terms $\alpha_n$ of the $\varepsilon$-expansion:

\begin{theorem}\label{thm:an}
  For $n =0,1,2,\ldots$
  \begin{align}\label{eq:alphan}
    \alpha_n = [\varepsilon^n] \;
    \pFq32{\frac{\varepsilon+2}2, \frac{\varepsilon+2}2,\frac{\varepsilon+2}2}{1,\frac{\varepsilon+3}2}{\frac14}
    &= (-1)^n \sum_{j=1}^\infty \frac{2}{j} \frac{1}{\binom{2j}{j}} \sum
    \prod_{k=1}^n \frac{A_{k,j}^{m_k}}{m_k! k^{m_k}}
  \end{align}
  where the inner sum is over all non-negative integers $m_1,\ldots,m_n$ such
  that $m_1+2m_2+\ldots+n m_n = n$.
\end{theorem}

\begin{proof}
  Equation \eqref{eq:alphan} may be derived  from \eqref{eq:hypeps} using, for
  instance, Fa\`{a} di Bruno's formula for the $n$-th derivative of the composition
  on two functions.
\end{proof}

\begin{example}[$\alpha_0$, $\alpha_1$ and $\alpha_2$]
  In particular,
  \begin{align*}
  \alpha_1 = [\varepsilon] \;
    \pFq32{\frac{\varepsilon+2}2, \frac{\varepsilon+2}2,\frac{\varepsilon+2}2}{1,\frac{\varepsilon+3}2}{\frac14}
    &= -\sum_{j=1}^\infty \frac{2}{j} \frac{1}{\binom{2j}{j}} A_{1,j} \\
    &= -\sum_{j=1}^\infty \frac{2}{j} \frac{1}{\binom{2j}{j}} \left[ \bar{S_1}-2S_1-1
    \right].
  \end{align*}
  Such \emph{multiple inverse binomial sums} are studied in
  \cite{davydychev-bin}. In particular, using \cite[(2.20),
  (2.21)]{davydychev-bin} we find
  \begin{align}
  \alpha_0 &= \frac{2\pi}{3\sqrt{3}}, \label{eq:eps0}\\
    \alpha_1 &= \frac{2}{3\sqrt{3}} \left[ \pi - \pi\log3 + \Ls{2}{\frac\pi3} \right].\label{eq:eps1}
  \end{align}

  For the second term $\alpha_2$ in the $\varepsilon$-expansion
  \eqref{eq:hypeps} produces
  \begin{align*}
    [\varepsilon^2] \;
    \pFq32{\frac{\varepsilon+2}2, \frac{\varepsilon+2}2,\frac{\varepsilon+2}2}{1,\frac{\varepsilon+3}2}{\frac14}
    &= \sum_{j=1}^\infty \frac{1}{j} \frac{1}{\binom{2j}{j}} \left[ A_{1,j}^2 + A_{2,j} \right] \\
    &= \sum_{j=1}^\infty \frac{1}{j} \frac{1}{\binom{2j}{j}} \left[ \bar{S}_2-S_2 + (\bar{S}_1-2S_1)^2 -2\bar{S}_1+4S_1 \right].
  \end{align*}
  Using \cite[(2.8),(2.22)-(2.24)]{davydychev-bin} we obtain
  \begin{align}\label{eq:eps2}
   \alpha_2
    &= \frac{2}{3\sqrt{3}} \left[ \frac{\pi}{72} - \pi\log3 + \frac12 \pi\log3
    + (1-\log3)\Ls{2}{\frac\pi3} \right. \nonumber\\
    &\quad\left.+ \frac32\Ls{3}{\frac\pi3}
    + \frac32\Ls{3}{\frac{2\pi}{3}} - 3\Ls{3}{\pi} \right].
  \end{align}
  \qede
\end{example}

These results provide us with evaluations of $\mu_1(1+x+y)$ and
$\mu_2(1+x+y)$ as given next. As expected, the result for $\mu_1(1+x+y)$ agrees
with Smyth's original evaluation, and the result for
$\mu_2(1+x+y)$ agrees with our prior evaluation in \cite{logsin1}.
The latter evaluation will be recalled in Section \ref{sec:mu2}.

\begin{theorem}[Evaluation of $\mu_1(1+x+y)$ and $\mu_2(1+x+y)$]\label{thm:mu21eps}
  We have
  \begin{align}
    \mu_1(1+x+y) &= \frac{1}{\pi} \Ls{2}{\frac\pi3}, \label{eq:mu1eps}\\
    \mu_2(1+x+y) &= \frac3\pi \Ls{3}{\frac{2\pi}{3}} + \frac{\pi^2}{4}. \label{eq:mu2eps}
  \end{align}
\end{theorem}

\begin{proof}
  Using Theorem \ref{thm:W3} we obtain
  \begin{align}
    \mu_1(1+x+y) = \frac{3\sqrt{3}}{2\pi} \left[ (\log3-1)\alpha_0 + \alpha_1 \right].
  \end{align}
  Combining this with equations \eqref{eq:eps0}, \eqref{eq:eps1} yields (\ref{eq:mu1eps}).

  Again, using Theorem \ref{thm:W3} we obtain
  \begin{align}
    \mu_2(1+x+y) = \frac{3\sqrt{3}}{2\pi} \left[ (\log^2 3 - 2\log3 + 2 - \tfrac{\pi^2}{12})\alpha_0
    + 2(\log3 - 1) \alpha_1 + 2 \alpha_2 \right].
  \end{align}
  Combining this with equations  \eqref{eq:eps0}, \eqref{eq:eps1}, and \eqref{eq:eps2} yields
  \begin{align}
    \pi\mu_2(1+x+y)
    &= 3\Ls{3}{\frac{2\pi}{3}} + 3\Ls{3}{\frac\pi3} - 6\Ls{3}{\pi} - \frac{\pi^3}{18} \nonumber\\
    &= 3\Ls{3}{\frac{2\pi}{3}} + \frac{\pi^3}{4}.
  \end{align}
  The last equality follows, for instance, automatically from the results in
  \cite{logsin-evaluations}.
\end{proof}

\begin{example}[$\alpha_3$]
  The evaluation of $\alpha_3$ is more involved and we omit some
  details. Again, \eqref{eq:hypeps} produces
  \begin{align*}
    [\varepsilon^3] \;
    \pFq32{\frac{\varepsilon+2}2, \frac{\varepsilon+2}2,\frac{\varepsilon+2}2}{1,\frac{\varepsilon+3}2}{\frac14}
    &= -\frac13 \sum_{j=1}^\infty \frac{1}{j} \frac{1}{\binom{2j}{j}}
    \left[ A_{1,j}^3 + 3A_{1,j}A_{2,j} + 2A_{3,j} \right].
  \end{align*}
  Using \cite[(2.25)-(2.28),(2.68)-(2.70),(2.81),(2.89)]{davydychev-bin} as well as results from \cite{logsin-evaluations} we obtain
  \begin{align}\label{eq:eps3}
    \alpha_3
    &= \frac{2}{3\sqrt{3}} \left[ \frac{5\pi^3}{108}(1-\log3) + \frac12 \pi\log^2 3
    - \frac16 \pi \log^3 3 + \frac{11}{9} \pi \zeta(3) \right. \nonumber\\
    &\quad + \Cl{2}{\frac\pi3} \left( \frac{5}{36} \pi^2 - \log3 + \frac12 \log^2 3 \right)
    - 3 \Gl{2,1}{\frac{2\pi}{3}} \left( 1 - \log3 \right) \nonumber\\
    &\quad\left.- \frac{35}{6} \Cl{4}{\frac\pi3} + 15 \Cl{2,1,1}{\frac{2\pi}{3}}
    - 3 \Lsc{2}{3}{\frac\pi3} \right].
  \end{align}
  Observe the occurrence of the log-sine-cosine integral
  $\Lsc{2}{3}{\frac\pi3}$. The log-sine-cosine integrals were defined
  in \eqref{eq:deflsc}.
  \qede
\end{example}

Proceeding as in the proof of Theorem \ref{thm:mu21eps} we obtain:

\begin{theorem}[Evaluation of $\mu_3(1+x+y)$]\label{thm:mu3eps}
  We have
  \begin{align}\label{eq:mu3eps}
    \pi\mu_3(1+x+y)
    &= 15\Ls{4}{\frac{2\pi}{3}} - 18\Lsc{2}{3}{\frac\pi3} - 15\Cl{4}{\frac\pi3} \nonumber\\
    &\quad- \frac14\pi^2\Cl{2}{\frac\pi3} - 17\pi\zeta(3).
  \end{align}
\end{theorem}

The log-sine-cosine integral $\Lsc{2}{3}{\frac\pi3}$ appears to reduce further as
\begin{align}
  12\Lsc{2}{3}{\frac\pi3}
  &\stackrel{?[1]}{=} 6\Ls{4}{\frac{2\pi}{3}} - 4\Cl{4}{\frac\pi3} - 7\pi\zeta(3) \\
  &= 6\Ls{4}{\frac{2\pi}{3}} - \frac89\Ls{4}{\frac\pi3} - \frac{59}{9}\pi\zeta(3). \label{eq:lsc23}
\end{align}
This conjectural reduction also appears in \cite[(A.30)]{dk1-eps} where it was
found via PSLQ.  Combining this with \eqref{eq:mu3eps}, we obtain an conjectural
evaluation of $\mu_3(1+x+y)$ equivalent to \eqref{eq:mu3ls}.

On the other hand, it follows from \cite[(2.18)]{davydychev-bin} that
\begin{align}
  12\Lsc{2}{3}{\frac\pi3}
  &= \Ls{4}{\frac{2\pi}{3}} - 4\Ls{4}{\frac\pi3} - \frac{1}{12}\pi\log^3 3 \nonumber\\
  &\quad+ 24\Ti{4}{\frac{1}{\sqrt{3}}} + 12\log3 \Ti{3}{\frac{1}{\sqrt{3}}} + 3\log^2 3 \Ti{2}{\frac{1}{\sqrt{3}}}.
\end{align}
Using the known evaluations --- see for instance \cite[(76),(77)]{logsin1} --- for
the inverse tangent integrals of order two and three, we find that
\eqref{eq:lsc23} is equivalent to
\begin{align}\label{eq:lsc23ti}
  \Ti{4}{\frac{1}{\sqrt{3}}}
  &\stackrel{?[1]}{=} \frac{5}{24}\Ls{4}{\frac{2\pi}{3}} + \frac{7}{54}\Ls{4}{\frac\pi3}
  - \frac{59}{216}\pi\zeta(3) - \frac{1}{288}\pi\log^3 3 \nonumber\\
  &\quad- \frac12\log3 \Ti{3}{\frac{1}{\sqrt{3}}} - \frac18\log^2 3 \Ti{2}{\frac{1}{\sqrt{3}}}.
\end{align}

\section[Trigonometric analysis of mun(1+x+y)]{Trigonometric analysis of $\mu_n(1+x+y)$}\label{sec:mun}

As promised in \cite{logsin1} --- motivated by the development outlined above
--- we take the analysis of $\mu_n(1+x+y)$ for $n \ge 3$ a fair distance. In
light of \eqref{eq:mu-trigintegral} we define
\begin{align}\label{eq:rho-n}
  \rho_n(\alpha) &:= \frac1{2\pi} \int_{-\pi}^{\pi}
    \left(\Re\log\left(1-\alpha\,\e^{i\,\omega}\right)\right)^n \id\omega
\end{align}
for $n \ge 0$ so that
\begin{align}\label{eq:mu-rhointegral}
  \mu_n(1+x+y) &= \frac{1}{2\pi} \int_{-\pi}^{\pi}\,\rho_n(|2 \sin \theta|) \id\theta.
\end{align}
We thus typically set $\alpha = |2\sin\theta|$. Note that $\rho_0(\alpha)=1$,
$\rho_1(\alpha)=\log(|\alpha| \vee 1))$.

\begin{proposition}[Properties of $\rho_n$]\label{prop:rho}
  Let $n =1,2, \ldots$.
  \begin{enumerate}[(a)]
    \item\label{prop:ra} For $|\alpha| \le 1$ we have
      \begin{align} \label{eq:zn}
        \rho_n(\alpha)  &=(-1)^n \sum_{m=1}^\infty\,\frac{\alpha^m}{m^n} \,\omega_n(m),
      \end{align}
      where $\omega_n$ is defined as
      \begin{align}\label{eq:cn}
        \omega_n(m) &=\sum_{\sum_{j=1}^n k_j=m}\,\frac{1}{2\pi}\,
        \int_{-\pi}^{\pi}\, \prod_{j=1}^n\,\frac{m}{k_j}\cos (k_j\omega) \id \omega.
      \end{align}
    \item\label{prop:rb} For $|\alpha| \ge 1$ we have
      \begin{align} \label{eq:mun}
        \rho_n(\alpha) &= \sum_{k=0}^n \binom{n}{k} \log^{n-k}|\alpha|\,\rho_k\left(\frac 1 \alpha \right).
      \end{align}
  \end{enumerate}
\end{proposition}

\begin{proof}
  For part \eqref{prop:ra} we use \eqref{eq:rho-n} to write
  \begin{align*}
    \rho_n(\alpha)
    &= \frac1{2\pi} \int_{-\pi}^{\pi}\, \left(\Re\,\log\left(1-\alpha\e^{i\,\omega}\right)\right)^n\id \omega\\
    &=\frac1{2\pi} \int_{-\pi}^{\pi}\, \left\{-\sum_{ k \ge 1}\frac{\alpha^k}{k}\,\cos (k\omega) \right\}^n\id \omega\\
    &=(-1)^n\sum_{m=1}^\infty\,\frac{\alpha^m}{m^n} \,\omega_n(m),
  \end{align*}
  as asserted. We note that $|\omega_n(m)| \le m^n$ and so the sum is convergent.

  For part \eqref{prop:rb} we now use \eqref{eq:rho-n} to write
  \begin{align*}
    \rho_n(\alpha)
    &= \frac{1}{2\pi} \int_{-\pi}^{\pi}\,\log^n \left(|\alpha|\,\left|1-\alpha^{-1}\e^{i\,\omega}\right| \right)\id\omega\\
    &= \frac{1}{2\pi} \int_{-\pi}^{\pi}\,\left( \log |\alpha|+\log\left|1-\alpha^{-1}\e^{i\,\omega}\right|\right)^n \id\omega\\
    &= \sum_{k=0}^n \binom{n}{k} \log^{n-k}|\alpha|\,\frac{1}{2\pi} \int_{-\pi}^{\pi}\,\log^k\left|1-\alpha^{-1}\e^{i\,\omega}\right| \id\omega\\
    &= \sum_{k=0}^n \binom{n}{k} \log^{n-k}|\alpha|\,\rho_k\left(\frac 1 \alpha \right),
  \end{align*}
  as required.
\end{proof}

\begin{example}[Evaluation of $\omega_n$ and $\rho_n$ for $n \le 2$]\label{ex:rho2}
  We have $\omega_0(m)= 0$, $\omega_1(m) = \delta_0(m)$, and
  \begin{align}\label{eqn:omega-ex}
    \omega_2(0) =1,\quad
    \omega_2(2m) = 2,\quad
    \omega_2(2m+1) = 0.
  \end{align}
  Likewise, $\rho_0(\alpha)= 1$, $\rho_1(\alpha) = \log\left(|\alpha| \vee 1\right)$, and
  \begin{align}
    \rho_2(\alpha)&= \frac12\Li_2(\alpha^2),  \qquad \mbox {for~} |\alpha| \le 1, \label{eqn:rho2}\\
    &= \frac12\Li_2\left(\frac{1}{\alpha^2}\right)+\log^2 |\alpha|,\qquad \mbox {for~} |\alpha| \ge 1,\label{eqn:rho-ex}
  \end{align}
  where the case $n=2$ follows from \eqref{eqn:omega-ex} and Proposition
  \ref{prop:rho}.
  \qede
\end{example}

We have arrived at an effective description of $\mu_n(1+x+y)$:

\begin{theorem}[Evaluation of $\mu_n(1+x+y)$]\label{thm:mun}
  Let $n =1,2, \ldots$. Then
  \begin{align}\label{eq:mugen}
    \mu_n(1+x+y) &= \frac{1}{\pi}\left\{ \Ls{n+1}{\frac{\pi}{3}}-\Ls{n+1}{\pi} \right\}
    + \frac{2}{\pi} \int_{0}^{\pi/6} \rho_n\left(2\sin\theta\right) \id\theta \nonumber\\
    &\quad+ \frac{2}{\pi} \sum_{k=2}^{n} \binom{n}{k} \int_{\pi/6}^{\pi/2}\,
    \log^{n-k}\,(2\sin\theta)\,\rho_k\left(\frac 1 {2\sin \theta} \right)\md \theta.
  \end{align}
\end{theorem}

\begin{proof}
  Since $|\alpha| <1$ exactly when $|\theta| < \pi/6$ we start with
  \eqref{eq:mu-rhointegral} to get
  \begin{align*}
    \mu_n(1+x+y) &= \frac{1}{2\pi} \int_{-\pi}^{\pi}\,\rho_n(|2 \sin \theta|) \id\theta \\
    &= \frac{2}{\pi} \int_{0}^{\pi/6}\,\rho_n(2 \sin \theta) \id\theta
    + \frac{2}{\pi} \int_{\pi/6}^{\pi/2}\,\rho_n(2 \sin \theta) \id\theta \\
    &= \frac{2}{\pi} \int_{0}^{\pi/6}\,\rho_n(2 \sin \theta) \id\theta \\
    &\quad+ \sum_{k=0}^n \binom{n}{k} \frac{2}{\pi} \int_{\pi/6}^{\pi/2}
    \log^{n-k}(2\sin \theta)\,\rho_k\left(\frac 1 {2\sin \theta} \right) \id\theta.
  \end{align*}
  We observe that for $k=1$ the contribution is zero since $\rho_1$ is zero
  for $|\alpha|<1$.  After evaluating the term with $k=0$ we arrive at
  \eqref{eq:mugen}.
\end{proof}

\begin{remark}\label{rem:sasaki}
  As is shown in \cite{logsin1},
  \begin{equation*}
    \frac{1}{\pi}\left\{ \Ls{n+1}{\frac{\pi}{3}} - \Ls{n+1}{\pi} \right\} = \mu_n(1+x+y_*)
  \end{equation*}
  is a multiple Mahler measure and is recursively evaluable for all $n$.
  \qede
\end{remark}

We note the occurrence of values at $\pi/3$ and so record the following:

\begin{example}[Values of $\Ls{n}{\pi/3}$]\label{ex:lspi3}
  The following evaluations may be obtained with the help of the implementation accompanying
  \cite{logsin-evaluations}.\footnote{These are available for download from \url{http://arminstraub.com/pub/log-sine-integrals}}
  \begin{align*}
    \Ls{2}{\frac\pi3} &= \mcl{2} \\
    -\Ls{3}{\frac\pi3} &= \frac{7}{108}\, \pi^3 \\
    \Ls{4}{\frac\pi3} &= \frac12\pi\,\zeta(3)+\frac 92\,\mcl{4} \\
    -\Ls{5}{\frac\pi3} &= \frac{1543}{19440}\pi^5 - 6\mgl{4,1} \\
    \Ls{6}{\frac\pi3} &= \frac{15}2\pi \,\zeta(5) + \frac{35}{36}\,\pi^3
      \zeta(3) + \frac{135}{2}\,\mcl{6} \\
    -\Ls{7}{\frac\pi3} &= \frac{74369}{326592}\pi^7 + \frac{15}{2}\,\pi
      \zeta(3)^2 - 135\,\mgl{6,1} \\
    \Ls{8}{\frac\pi3} &= \frac{13181}{2592}\pi^5\zeta(3) + \frac{1225}{24}\pi^3\zeta(5)
      + \frac{319445}{864}\pi\zeta(7) \\
      &\quad+ \frac{35}{2}\pi^2\mcl{6}
      + \frac{945}{4}\mcl{8} + 315\mcl{6,1,1}
  \end{align*}\qede
  \end{example}

\subsection[Further evaluation of rho]{Further evaluation of $\rho_n$}\label{ssec:rhon}

To make further progress, we need first to determine $\rho_n$ for $n \ge 3$.
It is instructive to explore the next few cases.

\begin{example}[Evaluation of $\omega_3$ and $\rho_3$]\label{ex:rho3}
  We use{\small \[4\,\cos \left( a \right) \cos \left( b \right) \cos \left( c \right) =
  \cos \left( a+b+c \right) +\cos \left( a-b-c \right) +\cos \left( a-b+
  c \right) +\cos \left( a-c+b \right)\]}
  and so derive
  \begin{align*}
    \omega_3(m) &= \frac 14\,\sum\, \left \{\frac{m^3}{ijk} \colon i\pm j\pm k =0, i+j+k=m\right\}.
  \end{align*}
  Note that we must have exactly one of $i=j+k, j=k+i$ or $k=i+j$.
  We thus learn that $\omega_3(2m+1)=0$. Moreover, by symmetry,
  \begin{align}\label{eq:omega3}
    \omega_3(2m) &= \frac34\,\sum_{j+k=m} \frac{(2m)^3}{jk(j+k)} \nonumber\\
    &= 6\,\sum_{j+k=m} \frac{m^2}{jk} = 12\,m\sum_{k=1}^{m-1} \frac 1k .
  \end{align}
  Hence, by Proposition \ref{prop:rho},
  \begin{align}\label{eq:rho3}
    \rho_3(\alpha) &= -\frac32\,\sum_{m=1}^\infty \frac{\sum_{k=1}^{m-1} \frac 1k}{m^2}\, \alpha^{2m}
    = -\frac32\,\Li_{2,1}(\alpha^2)
  \end{align}
  for $|\alpha|<1$.
\qede
\end{example}

\subsection[A general formula for rho]{A general formula for $\rho_n$}

In the general case we have
\begin{equation}\label{eq:prodcos}
  \prod_{j=1}^n \cos(x_j) = 2^{-n} \sum_{\varepsilon\in\{-1,1\}^n}
    \cos\left( \sum_{j=1}^n \varepsilon_j x_j \right)
\end{equation}
which follows inductively from $2\cos(a)\cos(b) = \cos(a+b)+\cos(a-b)$.

\begin{proposition}\label{prop:omegaodd}
  For integers $n,m\ge0$ we have $\omega_n(2m+1) = 0$.
\end{proposition}

\begin{proof}
  In light of \eqref{eq:prodcos} the summand corresponding to the indices
  $k_1,\ldots,k_n$ in \eqref{eq:cn} for $\omega_n(2m+1) = 0$ is nonzero if and
  only if there exists $\varepsilon\in\{-1,1\}^n$ such that $\varepsilon_1 k_1
  + \ldots + \varepsilon_n k_n = 0$.
  In other words, there is a set $S\subset\{1,\ldots,n\}$ such that
  \[ \sum_{j\in S} k_j = \sum_{j\not\in S} k_j. \]
  Thus $k_1+\ldots+k_n=2\sum_{j\in S} k_j$ which contradicts $k_1+\ldots+k_n=2m+1$.
\end{proof}

\begin{example}[Evaluation of $\omega_4$ and $\rho_4$]\label{ex:rho4}
  Proceeding as in Example \ref{ex:rho3} and employing \eqref{eq:prodcos}, we find
  \begin{align}\label{eq:o4}
    \omega_4(2m) &= \frac38 \sum_{\substack{i+j=m \\ k+\ell=m}} \frac{(2m)^4}{ijk\ell}
    + \frac12 \sum_{i+j+k=m} \frac{(2m)^4}{ijk\ell} \nonumber \\
    &= 24 m^2 \sum_{\substack{i<m \\ j<m}} \frac1{ij}
    + 24 m^2 \sum_{i+j<m} \frac1{ij} \nonumber\\
    &= 48 m^2 \sum_{i=1}^{m-1} \frac1i \sum_{j=1}^{i-1} \frac1j
    + 24 m^2 \sum_{i=1}^{m-1} \frac1{i^2}
    + 48 m^2 \sum_{i=1}^{m-1} \frac1i \sum_{j=1}^{i-1} \frac1j.
  \end{align}
  Consequently, for $|\alpha|<1$ and appealing to Proposition \ref{prop:rho},
  \begin{align}\label{eq:rho4}
    \rho_4(\alpha) = \sum_{m=1}^\infty \frac{\alpha^{2m}}{(2m)^4}\, \omega_4(2m)
     = 6\,\Li_{2,1,1}(\alpha^2) + \frac32\,\Li_{2,2}(\alpha^2).
  \end{align}
  This suggests that $\rho_n(\alpha)$ is generally expressible as a sum of
  polylogarithmic terms, as will be shown next.
  \qede
\end{example}

To help general evaluation of $\omega_n(2m)$, for integers $j\ge0$ and $m\ge1$,
let us define
\begin{equation}\label{eq:defsigma1}
  \sigma_j(m) :=\sum_{m_1+\ldots+m_j= m}\frac{1}{m_1\cdots m_j}.
\end{equation}

\begin{proposition}\label{prop:sigma}
  For positive integers $n$, $m$ we have
  \begin{equation}\label{eq:os}
    \frac{\omega_n(2m) }{m^n}= \sum_{j=1}^{n-1} \binom{n}{j} \sigma_{j}(m)\sigma_{n-j}(m)
  \end{equation}
  where $\sigma_{j}$ is as defined in \eqref{eq:defsigma1}.
\end{proposition}

\begin{proof}
  It follows from \eqref{eq:prodcos} that
  \begin{equation*}
    \omega_n(2m) = \sum_{k_1+\ldots+k_n=2m} \sum_{\stack{\varepsilon\in\{-1,1\}^n}{\sum_j \varepsilon_j k_j = 0}} \prod_{j=1}^n \frac{m}{k_j}.
  \end{equation*}
  Arguing as in Proposition \ref{prop:omegaodd} we therefore find that
  \begin{equation*}
    \omega_n(2m) = \sum_{j=1}^{n-1} \binom{n}{j} \sum_{\substack{k_1+\ldots+k_j=m \\ k_{j+1}+\ldots+k_n=m}} \prod_{j=1}^n \frac{m}{k_j}.
  \end{equation*}
  This is equivalent to \eqref{eq:os}.
\end{proof}

Moreover, we have a simple useful recursion:
\begin{proposition}\label{prop:srec}
  Let $m \ge 1$. Then $\sigma_{1}(m)=1/m$ while for $j \ge2$ we have
  \begin{align*}
    \sigma_{j}(m) = \frac{j}{m} \sum_{r=1}^{m-1} \sigma_{j-1}(r).
  \end{align*}
\end{proposition}

\begin{proof}
  We have
  \begin{align*}
    \sigma_j(m) &= \sum_{m_1+\ldots+m_j= m}\frac{1}{m_1\cdots m_j} \\
    % &= \frac{1}{m} \sum_{m_1+\ldots+m_j = m}\frac{m_1+\ldots+m_j}{m_1\cdots m_j} \\
    &= \frac{j}{m} \sum_{m_1+\ldots+m_j = m}\frac{1}{m_1\cdots m_{j-1}} \\
    &= \frac{j}{m} \sum_{r=1}^{m-1} \sum_{m_1+\ldots+m_{j-1} = r}\frac{1}{m_1\cdots m_{j-1}}
  \end{align*}
  which yields the claim.
\end{proof}

\begin{corollary}\label{cor:sigmajm}
  We have
  \begin{align*}
    \sigma_{j}(m) = \frac{j!}{m} \sum_{m>m_1>\ldots>m_{j-1}>0} \frac{1}{m_1 \cdots m_{j-1}}.
  \end{align*}
\end{corollary}

Thus, for instance, $\sigma_2(m)=2H_{m-1}/m$. From here, we easily re-obtain the
evaluations of $\omega_3$ and $\omega_4$ given in Examples \ref{ex:rho3} and
\ref{ex:rho4}. To further illustrate Propositions \ref{prop:sigma} and \ref{prop:srec}, we
now compute $\rho_5$ and $\rho_6$.

\begin{example}[Evaluation of $\rho_5$ and $\rho_6$]\label{ex:rho5}
  From Proposition \ref{prop:sigma},
  \begin{equation*}
    \frac{\omega_5(2m) }{m^5}= 10\sigma_1(m)\sigma_4(m) + 20\sigma_2(m)\sigma_3(m).
  \end{equation*}
  Consequently, for $|\alpha|<1$,
  \begin{align}\label{eq:rho5}
    -\rho_5(\alpha) &= \sum_{m=1}^\infty \frac{\alpha^{2m}}{(2m)^5}\, \omega_5(2m) \nonumber\\
    &= \frac{10\cdot4!}{32}\,\Li_{2,1,1,1}(\alpha^2)
    + \frac{20\cdot2!\cdot3!}{32}\left( 3\Li_{2,1,1,1}(\alpha^2) + \Li_{2,1,2}(\alpha^2) + \Li_{2,2,1}(\alpha^2) \right) \nonumber\\
    &= 30\,\Li_{2,1,1,1}(\alpha^2)
    + \frac{15}{2}\left(\Li_{2,1,2}(\alpha^2) + \Li_{2,2,1}(\alpha^2) \right).
  \end{align}
  Similarly, we have for $|\alpha|<1$,
  \begin{align}\label{eq:rho6}
    \rho_6(\alpha)\nonumber
    &= 180\,\Li_{2,1,1,1,1}(\alpha^2)
    + 45\left(\Li_{2,1,1,2}(\alpha^2) + \Li_{2,1,2,1}(\alpha^2) + \Li_{2,2,1,1}(\alpha^2) \right) \\
    &\quad+ \frac{45}{4}\Li_{2,2,2}(\alpha^2).
  \end{align}
  From these examples the general pattern, established next, begins to
  transpire.
  \qede
\end{example}

In general, $\rho_n$ evaluates as follows:

\begin{theorem}[Evaluation of $\rho_n$]\label{thm:rhon}
  For $|\alpha|<1$ and integers $n\ge2$,
  \begin{align*}
    \rho_n(\alpha) = \frac{(-1)^n n!}{4^n}
    \sum_w 4^{\ell(w)} \Li_{w}(\alpha^2)
  \end{align*}
  where the sum is over all indices $w=(2,a_2,a_3,\ldots,a_{\ell(w)})$ such
  that $a_2,a_3,\ldots\in\{1,2\}$ and $|w|=n$.
\end{theorem}

\begin{proof}
  From Proposition \ref{prop:sigma} and Corollary \ref{cor:sigmajm} we have
  \begin{align*}
    \rho_n(\alpha) = \frac{(-1)^n n!}{2^n} \sum_{m=1}^\infty \frac{\alpha^{2m}}{m^2}
    \sum_{j=0}^{n-2} \sum_{\stack{m>m_1>\ldots>m_j>0}{m>m_{j+1}>\ldots>m_{n-2}>0}}
    \frac{1}{m_1\cdots m_{n-2}}.
  \end{align*}
  Combining the right-hand side into harmonic polylogarithms yields
  \begin{align*}
    \rho_n(\alpha) = \frac{(-1)^n n!}{2^n}
    \sum_{k=0}^{n-2} \sum_{\stack{a_1,\ldots,a_k\in\{1,2\}}{a_1+\ldots+a_k=n-2}}
    2^{c(a)} \Li_{2,a_1,\ldots,a_k}(\alpha^2)
  \end{align*}
  where $c(a)$ is the number of $1$s among $a_1,\ldots,a_k$.
  The claim follows.
\end{proof}

\begin{example}[Special values of $\rho_n$]\label{ex:rho-vals}
  Given Theorem \ref{thm:rhon},  one does not expect to be able to evaluate
  $\rho_n(\alpha)$ explicitly at most points.  Three exceptions are $\alpha=0$
  (which is trivial), $\alpha=1$, and $\alpha=1/\sqrt{2}$.  For instance we
  have $\rho_4(1)= \frac{19}{240}\,\pi^4$ as well as $-\rho_5(1) =
  \frac{45}{2}\zeta(5) + \frac{5}{4}\zeta(3)$ and $\rho_6(1) =
  \frac{275}{1344}\pi^6 + \frac{45}{2}\zeta(3)^2$.  At $\alpha= 1/\sqrt{2}$ we have
  \begin{align}
    \rho_4\left(\frac1{\sqrt{2}}\right)
    &= \frac{7}{16}\log^4 2 + \frac{3}{16}\pi^2 \log^2 2 - \frac{39}{8}\zeta(3) \log 2
    + \frac{13}{192}\pi^4 - 6\Li_4\left(\tfrac12\right).
  \end{align}
  For $n \ge 5$ the expressions are expected to be more complicated.
  \qede
\end{example}

\subsection{Reducing harmonic polylogarithms of low weight}\label{sec:lireductions}

Theorems \ref{thm:mun} and \ref{thm:rhon} take us closer to a closed form for $\mu_n(1+x+y)$. As $\rho_n$ are expressible in terms of multiple harmonic polylogarithms of weight $n$, it remains to supply reductions those of low
weight. Such polylogarithms are reduced by the use of the differential operators
\begin{equation*}
  (D_0 f)(x)= x f'(x) \quad \mbox{and} \quad (D_1 f)(x)=(1-x)f'(x)
\end{equation*}
depending on whether the outer index is `$2$' or `$1$' \cite{b3l}.

\begin{enumerate}
  \item As was known to Ramanujan, and as studied further in
    \cite[\S8.1]{goldbach}, for $0<x<1$,
    \begin{align}\label{eq:z21}
      \Li_{2,1}(x) &= \frac12\, \log^2(1-x)\log(x)+ \log(1-x)\Li_2(1-x) \nonumber\\
      &\quad -\Li_3(1-x) + \zeta(3). %\\
      % &= -\Li_{2,1}(-x) +\frac 14\, \Li_{2,1}\left(x^2\right) +\Li_{2,1}\left(\frac{2x}{x+1}\right)
      % -\frac18\,\Li_{2,1}\left(\frac{4x}{(x+1)^2}\right).\nonumber
    \end{align}
    Equation \eqref{eq:z21}, also given in \cite{lewin2}, provides a useful
    expression numerically and symbolically.  For future use, we also record
    the relation, obtainable as in \cite[\S6.4 \& \S6.7]{lewin2},
    \begin{align}\label{eq:z21r}
      \Re\Li_{2,1}\left(\frac{1}{x}\right) + \Li_{2,1}(x)
      &= \zeta(3) -\frac16 \log^3 x + \frac12\pi^2 \log x \nonumber\\
      &\quad -\Li_2(x) \log x + \Li_3(x) \qquad \mbox{for~} 0 < x < 1.
    \end{align}
  \item For $\Li_{2,2}$ we work as follows.
    As $(1-x) \Li'_{1,2}(x) =\Li_2 \left(x \right) $, integration yields:
    \begin{align}
      \Li_{1,2}(x) &= 2\Li_3(1-x)-\log(1-x)\Li_2(x) -2\log(1-x)\Li_2(1-x) \nonumber\\
      &\quad- \log(1-x)^2\log(x)-2\zeta(3).
    \end{align}
    Then, since $x\Li'_{2,2}(x) =\Li_{1,2}(x)$, on integrating again we obtain
    $\Li_{2,2}(x)$ in terms of polylogarithms up to order four.  We appeal to
    various formulae in \cite[\S6.4.4]{lewin2} to arrive at
    \begin{align*}\Li_{2,2}(t) &=\nonumber \frac 12\,
      \log^2 (1-t) \log^2 t
      -2\zeta(2)\log (1-t) \log t -2\zeta(3) \log t -\frac12\,\Li_2^2 (t) \\& +2\Li_3 \left(1-t
      \right)\log t -2 \int_{0}^{t}\!{\frac{ \Li_2 \left(x \right)\log x
      }{1-x}}{\id x}-\int_{0}^{t}\!{\frac{\log \left( 1-x \right)\log^2 x}{1-x}}{\md x}.
    \end{align*}
    Expanding the penultimate integral as a series leads to
    \[\int_{0}^{t}\!{\frac{ \Li_2 \left(x \right)\log x
    }{1-x}}{\id x}=\Li_{1,2}(t) \log t-\Li_{2,2}(t),\]
    and so to
    \begin{align}
      \Li_{2,2}(t) &= \frac 12 \Li_2^2 (t)-2\zeta(3)\ln t +2 \Li_3 (1-t)\log t-\frac 12\,
      \log^2 (1-t) \log^2 t \nonumber\\
      &\quad- 2 \log (1-t) \Li_2 (1-t) \log t+\int_{0}^{t}\!\frac{\log\left( 1-x \right) \log^2 x}{1-x}{\id x}.
    \end{align}
    Now \cite[A3.4 Eq. (12)]{lewin2} will evaluate the final integral and we deduce:
    \begin{align}
      \Li_{2,2}(t) &= -\frac1{12}\, \log^4(1-t)+ \frac13 \log^3(1-t) \log t -\zeta(2)  \log^2 (1-t)  \nonumber\\
      &\quad+ 2\log(1-t) \Li_3 (t) -2\,\zeta ( 3) \log(1-t) -2\,\Li_4(t) \nonumber \\
      &\quad- 2\Li_4 \left({\frac{t}{t-1}} \right) +2\,\Li_4(1-t) -2\zeta(4) +\frac12\, \Li_2^2 (t).
    \end{align}
  \item The form for $\Li_{3,1}(t)$ is obtained in the same way but starting
    from $\Li_{2,1}(t)$ as given in \eqref{eq:z21}.  This leads to:
    \begin{align}
      2\,\Li_{3,1}(t)+\Li_{2,2}(t)= \frac12\, \Li_2^2(t).
    \end{align}
    This symmetry result, and its derivative
    \begin{align}
      2\,\Li_{2,1}(t)+\Li_{1,2}(t)= \Li_1(t) \Li_2(t),
    \end{align}
    are also obtained in \cite[Cor. 2 \& Cor. 3]{zlobin} by other methods.
  \item Since $\Li_{2,1,1}(x) =\int_0^x\Li_{1,1,1}(t)/t\id t$  and
    $\Li_{1,1,1}(x) =\int_0^x\Li_{1,1}(t)/(1-t)\id t$, we first compute
    $\Li_{1,1}(x)=\log^2(1-x)/2$ and so $\Li_{1,1,1}(x)=-\log^3(1-x)/6$ (the
    pattern will be clear). Hence
    \begin{align}
      \Li_{2,1,1}(x) &=-\frac16 \int_0^x\,\log^3(1-t)\frac{\md t}t \nonumber\\
      &=\frac{\pi^4}{90}-\frac16 \log(1-t)^3 \log t - \frac12 \log(1-t)^2 \Li_2(1-t) \nonumber\\
      &\quad+ \log(1-t)\Li_3(1-t)-\Li_4(1-t).
    \end{align}
  \item In general,
    \begin{equation}
      \Li_{\{1\}_n}(x) = \frac{(-1)^n}{n!} \log(1-x)^n,
    \end{equation}
    and therefore
    \begin{align}
      \Li_{2,\{1\}_{n-1}}(x) &= \frac{(-1)^n}{n!} \int_0^x \log(1-t)^n \frac{\md t}{t} \nonumber \\
      &= \zeta(n+1)-\sum_{m=0}^n \frac{(-1)^{n-m}}{(n-m)!} \log(1-x)^{n-m} \Li_{m+1}(1-x).
    \end{align}
\end{enumerate}

We have, inter alia, provided closed reductions for all multiple polylogarithms
with weight less than five. One does not expect such complete results
thereafter.

The reductions presented in this section allow us to express $\rho_3$ and
$\rho_4$ in terms of polylogarithms of depth $1$. Equation \eqref{eq:z21}
treats $\rho_3$ while \eqref{eq:rho4} leads to
\begin{align}\label{eq:rho4b}
  \rho_4\left(\alpha^2\right)&=3\left(\Li_3 \left(\alpha^{2} \right)-\zeta(3) +\Li_3
  \left( 1-{\alpha}^{2}\right) \right) \log \left( 1-{\alpha}^{2} \right)-\frac18\,\log^4 \left(1 -\alpha^{2} \right) \nonumber \\
  &\quad+ 3\zeta(4)- 3\,\Li_4\left(\frac{-\alpha^2}{1-\alpha^2} \right)- 3\,\Li_4\left(\alpha^{2} \right)-
  3\,\Li_4\left(1-\alpha^{2} \right)+\frac34\Li_2^2\left(1-\alpha^{2} \right) \nonumber\\
  &\quad- \log \alpha \,\log^3 \left( 1-\alpha^{2}\right)
  -\left( \frac{\pi^2}4+3\Li_2 \left( 1-\alpha^{2}  \right)\right)\log^2 \left( 1-\alpha^{2} \right).
\end{align}

\section[Explicit evaluations of mun(1+x+y) for small n]{Explicit evaluations of $\mu_n(1+x+y)$ for small
 $n$}\label{sec:mu23}

We now return to the explicit evaluation of the multiple Mahler measures
$\mu_k(1+x+y)$. The starting point for this section is the evaluation of
$\mu_2(1+x+y)$ from \cite{logsin1} which is reviewed in Section \ref{sec:mu2}
and was derived alternatively in Theorem \ref{thm:mu21eps}. Building on this,
we then present an informal evaluation of $\mu_3(1+x+y)$ in Section
\ref{sec:mu3}. A conjectural evaluation of $\mu_4(1+x+y)$ is presented in
equation \eqref{eq:mu4} of the Conclusion.

\subsection[Evaluation of mu2(1+x+y)]{Evaluation of $\mu_2(1+x+y)$}\label{sec:mu2}

\begin{theorem}[Evaluation of $\mu_2(1+x+y)$]\label{thm:mu21xy}
  We have
  \begin{align}\label{eq:mu2ls}
    \mu_2(1+x+y) = \frac3\pi \Ls{3}{\frac{2\pi}{3}} + \frac{\pi^2}{4}.
  \end{align}
\end{theorem}

By comparison, Smyth's original result may be written as (see \cite{logsin1})
\begin{equation}\label{eq:mu1xy}
  \mu_1(1+x+y) = \frac{3}{2\pi} \Ls{2}{\frac{2\pi}{3}}
  = \frac{1}{\pi}\mcl{2}.
\end{equation}

% In terms of the righthand side of \eqref{eq:singleh} with $n=1$, this is equivalent to the series evaluation:
% \[\frac{3\sqrt{3}}{8}\,\sum_{n=0}^{\infty }\left( 3\sum_{k=1}^n\frac 1k-2\sum_{k=0}^n\frac 1{2k+1}\right)
% \frac{ 4^n}{\binom{2n}{n}(2n+1)}
% =\mcl{2} -\pi \log 3.\]

We recall from \cite{logsin1} that the evaluation in Theorem
\ref{thm:mu21xy} proceeded by first establishing the following dilogarithmic
form:

\begin{proposition}[A dilogarithmic representation]\label{prop:mu2}
  We have:
  \begin{enumerate}[(a)]
    \item\label{propa}
      \begin{equation}\label{eq:z2}
        \frac 2\pi \int_{0}^{\pi} \Re\Li_2\left(4\sin^2\theta\right) \md\theta = 2\zeta(2).
      \end{equation}
     \item\label{propb}
       \begin{equation}\label{eq:mu2}
        \mu_2(1+x+y) =
        \frac{\pi^2}{36}+\frac 2\pi \int_{0}^{\pi/6} \Li_2\left(4\sin^2\theta\right) \md\theta.
      \end{equation}
  \end{enumerate}
\end{proposition}

We include the proof from \cite{logsin1} as it is instructive for evaluation of
$\mu_3(1+x+y)$.

\begin{proof}
  For \eqref{propa} we define $\tau(z):= \frac 2\pi \int_{0}^{\pi}
  \Li_2\left(4z\sin^2\theta\right) \md\theta.$ This is an analytic function of
  $z$. For $|z|<1/4$ we may use the original series for $\Li_2$ and expand
  term by term using Wallis' formula to derive
  \begin{align*}
    \tau(z) &= \frac 2\pi\,\sum_{n\ge 1}\frac{(4z)^n}{n^2}\int_{0}^{\pi}\,\sin^{2n}\theta\id\theta
    = 4z\,\pFq43{1, 1, 1, \frac 32}{2, 2, 2}{4z} \\
    &= 4\Li_2\left( \frac12-\frac12\sqrt{1-4z} \right) - 2 \log\left( \frac12+\frac12\sqrt{1-4z} \right)^2.
  \end{align*}
  The final equality can be obtained in \emph{Mathematica} and then verified by
  differentiation.  In particular, the final function provides an analytic
  continuation and so we obtain $\tau(1)=2\zeta(2) + 4i\mcl{2}$ which yields
  the assertion.

  For \eqref{propb}, commencing much as in \cite[Thm. 11]{klo}, we write
  \begin{equation*}
    \mu_2(1+x+y) = \frac{1}{4\pi^2} \int_{-\pi}^{\pi} \int_{-\pi}^{\pi}
    \Re\log\left(1-2\sin(\theta)\e^{i\,\omega}\right)^2 \id\omega\id\theta.
  \end{equation*}
  We consider the inner integral
  $\rho(\alpha):=\int_{-\pi}^{\pi}\left(\Re\log\left(1-\alpha\,\e^{i\,\omega}\right)\right)^2\id \omega$ with $\alpha := 2\sin\theta$.  For
  $|\theta| \le \pi/6$ we directly apply Parseval's identity to obtain
  \begin{equation}\label{eq:mrhoA}
    \rho(2\sin \theta) = \pi \Li_2\left(4\sin^2\theta\right).
  \end{equation}
  which is equivalent to \eqref{eqn:rho2} since $\rho(\alpha) =
  2\pi\rho_2(\alpha)$.  In the remaining case we write
  \begin{align}\label{eq:mrhoB}
    \rho(\alpha) &=
    \int_{-\pi}^{\pi}\left\{\log|\alpha| + \Re\log \left(1-\alpha^{-1}\,\e^{i\,\omega}\right)\right\}^2 \id\omega \nonumber\\
    &= 2\pi\,\log^2 |\alpha|- 2\log|\alpha|\int_{-\pi}^{\pi}\log \left|1-\alpha^{-1}\,\e^{i\,\omega}\right|\id \omega
    + \pi \Li_2\left(\frac{1}{\alpha^2}\right) \nonumber\\
    &= 2\pi\,\log^2 |\alpha| + \pi \Li_2\left(\frac{1}{\alpha^2}\right),
  \end{align}
  where we have appealed to Parseval's and Jensen's formulae.  Thus,
  \begin{equation}\label{eq:mrhoC}
    \mu_2(1+x+y) = \frac 1\pi \int_{0}^{\pi/6} \Li_2\left(4\sin^2\theta\right) \md\theta
    + \frac 1\pi \int_{\pi/6}^{\pi/2} \Li_2\left(\frac{1}{4\sin^2\theta}\right) \md\theta
    + \frac{\pi^2}{54},
  \end{equation}
  since $\frac 2\pi\int_{\pi/6}^{\pi/2}\log^2 \alpha\id\theta = \mu_2(1+x+y_*) =
  \frac{\pi^2}{54}$.  Now, for $\alpha >1$, the functional equation in
  \cite[A2.1 (6)]{lewin}
  \begin{align}\label{eq:dilog}
    \Li_2 (\alpha) +\Li_2 (1/\alpha) +\frac 12\,\log^2 \alpha &=2\zeta(2)+ i\pi \log \alpha
  \end{align}
  gives:
  \begin{equation}\label{eq:fed1}
    \int_{\pi/6}^{\pi/2}\left\{\Re\Li_2 \left(4 \sin^2 \theta \right)
    + \Li_2 \left(\frac1{4 \sin^2 \theta} \right)\right\}\id\theta
    = \frac{5}{54} \pi^3.
  \end{equation}
  We now combine \eqref{eq:z2}, \eqref{eq:fed1} and \eqref{eq:mrhoC} to deduce
  the desired result in \eqref{eq:mu2}.
\end{proof}

\subsection[Evaluation of mu3(1+x+y)]{Evaluation of $\mu_3(1+x+y)$}\label{sec:mu3}

In this section we provide a remarkably concise closed form of
$\mu_3(1+x+y)$.  We were led to this form by the integer relation algorithm
PSLQ \cite{bbg} (see Example \ref{ex:num} for some comments on obtaining high
precision evaluations), and by considering the evaluation \eqref{eq:mu2ls} of
$\mu_2(1+x+y)$.

The details of formalization are formidable --- at least by the route
chosen here --- and so we proceed more informally leaving three conjectural identities.

\begin{conjecture}[Evaluation of $\mu_3(1+x+y)$]\label{thm:mu3f2} We have
\begin{align}\label{eq:mu3ls}
  \mu_3(1+x+y) \stackrel{?[1]}{=} \frac{6}{\pi}\Ls{4}{\frac{2\pi}{3}} - \frac9\pi\mcl{4}
  - \frac\pi4\mcl{2} - \frac{13}{2}\zeta(3).
\end{align}
\end{conjecture}

This evaluation is equivalent to the conjectural identities \eqref{eq:lsc23}
and \eqref{eq:lsc23ti}.

\begin{proof}
We first use Theorem \ref{thm:mun} to write
\begin{align}\label{eq:mu3}
  \mu_3(1+x+y) &= \frac{2}{\pi} \int_{0}^{\pi/6} \rho_3(2\sin\theta) \id\theta
  + \frac{2}{\pi} \int_{\pi/6}^{\pi/2} \rho_3\left(\frac{1}{2\sin\theta}\right) \id\theta \\
  &\quad+ \frac{3}{\pi} \int_{\pi/6}^{\pi/2} \log(2\sin\theta)\,\Li_2\left(\frac 1 {4\sin^2 \theta} \right) \id\theta
  - \zeta(3) + \frac{9}{2\pi}\mcl{4}, \nonumber
\end{align}
on appealing to Examples \ref{ex:pi}  and \ref{ex:lspi3}.

Now the functional equation for the dilogarithm \eqref{eq:dilog} as used above
and knowledge of $\Ls{n}{\pi/3}$ (see \cite{logsin1,logsin-evaluations}) allow
us to deduce
\begin{align}
  \frac3\pi \int_{0}^{\pi/6} \log(2\sin\theta)&\,\Li_2\left(4\sin^2 \theta \right) \md\theta
  + \frac3\pi \int_{0}^{\pi/6} \log(2\sin\theta)\,\Li_2\left(\frac{1}{4\sin^2 \theta}\right) \md\theta \nonumber\\
  &= \frac32\zeta(3)-\frac{\pi}2 \mcl{2}+\frac{27}{2\pi}\mcl{4}, \label{eq:pi6}\\
  \frac3\pi \int_{\pi/6}^{\pi/2} \log(2\sin\theta)&\,\Re\Li_2\left(4\sin^2 \theta\right) \md\theta
  + \frac3\pi \int_{\pi/6}^{\pi/2} \log(2\sin \theta)\,\Re\Li_2\left(\frac{1}{4\sin^2 \theta}\right) \md\theta \nonumber\\
  &= 3\zeta(3)+\frac{\pi}2\,\mcl{2}-\frac{27}{2\pi}\mcl{4} \label{eq:pi6b}
\end{align}
Moreover, we have
\begin{align}
  \frac3\pi \left\{\int_{0}^{\pi/6} + \int_{\pi/6}^{\pi/2}\right\}
  \log(2\sin\theta)\,\Re\Li_2\left(4\sin^2 \theta \right) \md\theta
  &\stackrel{?[2]}{=} \frac72\zeta(3) - \pi\mcl{2}, \label{eq:pi2a} \\
  \frac3\pi \left\{\int_{0}^{\pi/6} + \int_{\pi/6}^{\pi/2}\right\}
  \log(2\sin\theta)\,\Re\Li_2\left(\frac{1}{4\sin^2 \theta}\right) \md\theta
  &\stackrel{?[2]}{=} \zeta(3) + \pi\mcl{2}, \label{eq:pi2b}
\end{align}
which are provable as was \eqref{eq:z2} because for $|z| < 1/2$ we have
\[ \frac 1\pi\, \int_{0}^{\pi }\!\log \left( 2\,\sin \frac\theta2 \right) \Li_2\left(4\,{z}^{2} \sin^2 \frac\theta2\right) \id\theta
=\sum_{n=1}^{\infty } \binom{2n}{n} \frac{\sum_{k=1}^{2n} \frac{(-1)^k}{k}}{n^2}\,z^{2n}. \]
(The latter is derivable also from \eqref{eq:pi6}, \eqref{eq:pi6b} and \eqref{eq:pi2a}.)

Thence, \eqref{eq:pi6}, \eqref{eq:pi6b} and \eqref{eq:pi2a} together establish
that the equality
\begin{align}\label{eq:mu3f}
  \frac3\pi \int_{\pi/6}^{\pi/2}\log(2\sin\theta)\,\Li_2\left(\frac{1}{4\sin^2 \theta} \right) \md\theta
  &=\frac23\zeta(3)+\frac{7\pi}{12}\mcl{2}-\frac{17}{2\pi}\mcl{4}
\end{align}
is true as soon as we establish
\begin{align}\label{eq:I3}
  I_3 := \frac3\pi \int_{0}^{\pi/6} \log(2\sin\theta)
  \Li_2\left(4\sin^2 \theta \right) \md\theta
 \stackrel{?[3]}{=}\frac76\,\zeta(3)-\frac{11\pi}{12}\,\mcl{2}+5\mcl{4}.
\end{align}
This can be achieved by writing the integral as
\begin{align*}
  I_3 &= \frac3\pi\,\sum_{n=1}^\infty\frac 1{n^2}  \int_0^1\!\frac{ s^{2n}}{\sqrt {4-{s}^{2}}}\log s\,{\md s}
\end{align*}
and using the binomial series to arrive at
\begin{align}\label{3sum}
  I_3&=-\frac{3}{2\pi}\,\sum_{m=0}^\infty\frac{\binom{2m}{m}}{4^{2m}} \sum_{n=1}^\infty\frac 1{n^2 \left(1+2(n+m)\right)^2}.
\end{align}
Now \eqref{3sum} can in principle be evaluated to the asserted result
in \eqref{eq:I3}; so this leaves us to deal with the two terms involving
$\rho_3$.

These two terms are in turn related by
\begin{align}
  \frac2\pi \int_{0}^{\pi/6}\Li_{2,1}\left(4\sin^2 \theta \right) \md\theta
  &+ \frac2\pi \int_{0}^{\pi/6} \Re\Li_{2,1}\left(\frac{1}{4\sin^2 \theta} \right) \md\theta \nonumber\\
  &= \frac1{9}\left\{\zeta(3) - \pi\mcl{2} + \frac{6}{\pi}\mcl{4}\right\}, \label{eq:mu36}
\end{align}
as we see by integrating \eqref{eq:z21r}. Likewise,
\begin{align}
  \frac2\pi \int_{\pi/6}^{\pi/2} \Re\Li_{2,1}\left(4\sin^2 \theta \right) \md\theta
  &+ \frac2\pi \int_{\pi/6}^{\pi/2} \Li_{2,1}\left(\frac{1}{4\sin^2 \theta} \right) \md\theta \nonumber\\
  &= \frac1{9}\left\{2\zeta(3) - 5\pi\mcl{2} - \frac{6}{\pi}\mcl{4}\right\}. \label{eq:mu36b}
\end{align}
Also, using \eqref{eq:z21} we arrive at
\begin{align}\label{eq:rint}
  \frac2\pi \int_{0}^{\pi/6} \Li_{2,1}\left(4\sin^2 \theta \right) \md\theta
  &= \frac{20}{27}\zeta(3) - \frac{8\pi}{27}\mcl{2} + \frac{4}{9\pi}\mcl{4} \nonumber\\
  &\quad+ \frac1\pi \int_{0}^{\pi/3} \log^2 \left(1-4\sin^2\frac\theta 2 \right)
  \log \left( 2\sin\frac\theta 2 \right) \md\theta,
\end{align}
and
\begin{equation}\label{eq:rintb}
  \frac2\pi \int_{0}^{\pi/2} \Re\Li_{2,1}\left(4\sin^2 \theta \right)\md \theta
  = \frac{1}{3}\zeta(3) - \frac{2\pi}{3}\mcl{2}.
\end{equation}

We may now establish --- from \eqref{eq:mu3f}, \eqref{eq:mu36},
\eqref{eq:mu36b}, \eqref{eq:rint}, \eqref{eq:rintb} and \eqref{eq:mu3} ---
that
\begin{align}
  \mu_3(1+x+y) &= \frac{43}{18}\,\zeta(3)-\frac{47\pi}{36}\,\mcl{2}-\frac{13}{3\pi}\,\mcl{4}  \nonumber \\
  &\quad+ \frac{2}\pi\int_{0}^{\pi/3} \log^2 \left(1-4\sin^2\frac\theta 2 \right)
  \log \left( 2\sin\frac\theta2 \right) \md\theta.
\end{align}
Hence, to prove \eqref{eq:mu3ls} we are reduced to verifying that
\begin{align}
  - \frac{1}\pi \Ls{4}{\frac{2\pi}3}
  &\stackrel{?[4]}{=} -\frac{37}{54}\zeta(3) + \frac{7\pi}{27}\mcl{2} - \frac7{9\pi}\mcl{4}  \nonumber\\
  &\quad+ \frac{1}{2\pi} \int_{0}^{\pi/3} \log^2 \left(1-4\sin^2\frac\theta 2 \right)
  \log \left( 2\sin\frac\theta 2 \right) \md\theta.
\end{align}
which completes the evaluation.
\end{proof}

\begin{remark}
  By noting that, for integers $n\ge2$,
  \begin{equation*}
    \Cl{n}{\frac\pi3} = \left( \frac{1}{2^{n-1}} + (-1)^n \right) \Cl{n}{\frac{2\pi}{3}},
  \end{equation*}
  the arguments of the Clausen functions in the evaluation \eqref{eq:mu3ls} of
  $\mu_3(1+x+y)$ may be transformed to $\frac{2\pi}{3}$.

  Many further variations are possible. For instance, it follows from \cite{logsin-evaluations} that
  \begin{align}
    \Ls{4}{\frac{2 \pi} 3}
    & = \frac{31}{18}\,\pi\zeta(3)+\frac{\pi^2}{12}\Cl{2}{\frac{2\pi}3}
    -\frac32\Cl{4}{\frac{2\pi}3}+6\Cl{2,1,1}{\frac{2\pi}3}
  \end{align}
  in terms of multi Clausen values.
  \qede
\end{remark}

\section{Proofs of two conjectures of Boyd}\label{sec:boyd}

We now use log-sine integrals to recapture the following evaluations
conjectured by Boyd in 1998 and first proven in \cite{two} using \emph{Bloch-Wigner}
logarithms.  Below ${\rm L}_{-n}$ denotes a primitive $\rm L$-series and $\G$
is Catalan's constant.

\begin{theorem}[Two quadratic evaluations]\label{thm:boyd} We have
 \begin{align}
    \mu(y^2(x+1)^2+y(x^2+6x+1)+(x+1)^2)
    &= \frac{16}{3\pi}\,{\rm L}_{-4}(2) = \frac{16}{3\pi}\,\G, \label{ex:bl1}\\
    \intertext{as well as}
    \mu(y^2(x+1)^2+y(x^2-10x+1)+(x+1)^2)
    &= \frac{5\sqrt{3}}{\pi}\,{\rm L}_{-3}(2) = \frac{20}{3\,\pi}\mcl{2}. \label{ex:bl2}
  \end{align}
\end{theorem}

\begin{proof}
  Let $P_c = y^2(x+1)^2+y(x^2+2cx+1)+(x+1)^2$ and $\mu_c = \mu(P_c)$ for a real
  variable $c$.  We set $x=\e^{2\pi it}$, $y=\e^{2\pi iu}$ and note that
  \begin{align*}
    |P_c| &= |(x+1)^2(y^2+y+1)+2(c-1)xy| \\
    &= \left|(x+x^{-1}+2)(y+1+y^{-1})+2(c-1)\right| \\
    &= |2(\cos(2\pi t)+1)(2\cos(2\pi u)+1)+2(c-1)| \\
    &= 2\,|c+2\cos(2\pi u)+(1+2\cos(2 \pi u))\cos(2\pi t)|.
  \end{align*}
  It is known that (see \cite[\S4.224, Ex. 9]{gr}), for real $a, \,b$ with $|a|\ge |b|>0,$
  \begin{align}\label{eq:lab}
    \int_0^1 \log\left|2a+2b\cos(2\pi \theta)\right| \id\theta
    = \log\left( |a|+\sqrt{a^2-b^2}\right).
  \end{align}
  Applying this, with $a=c+2\cos(2\pi u)$ and $b=1+2\cos(2 \pi u))$ to $\int_0^1 |P_c| \id t,$ we get
  \begin{equation}\label{eq:mc}
    \mu_c = \int_0^1 \log\left|c+2\cos(2\pi u) + \sqrt{({c}^{2}-1)+4(c-1) \cos(2\pi u)}\right| \id u.
  \end{equation}
  If $c^2-1=\pm 4(c-1)$, that is if $c=3$ or $c=-5$, then the surd is a perfect square and also $|a|\ge |b|.$

  (a) When $c=3$ for (\ref{ex:bl1}), by symmetry, after factorization we obtain
  \begin{align*}
    \mu_3 &= \frac 1\pi \int_0^\pi \log(1+4|\cos \theta|+4|\cos^2 \theta|) \id\theta
    = \frac4\pi \int_0^{\pi/2} \log(1+2\cos\theta) \id\theta \\
    &= \frac4\pi \int_0^{\pi/2} \log\left(\frac{2\sin\frac{3\theta}2}{2\sin\frac{\theta}2}\right)\id\theta
    = \frac4{3\,\pi}\,\left(\Ls{2}{\frac{3\pi}2}-3\Ls{2}{\frac{\pi}2}\right) \\
    &= \frac{16}{3}\,\frac{{\rm L}_{-4}(2)}{\pi}
  \end{align*}
  as required, since $\Ls{2}{\frac{3\pi}2}=-\Ls{2}{\frac{\pi}2}= {\rm L}_{-4}(2)$, which is Catalan's constant $\G$.

  (b) When $c=-5$ for (\ref{ex:bl2}), we likewise obtain
  \begin{align*}
    \mu_{-5} &= \frac 2\pi \int_0^\pi \log\left(\sqrt 3+2\sin\theta\right) \id\theta
    = \frac2\pi \int_{\pi/3}^{4\pi/3} \log\left(\sqrt 3+2\sin\left(\theta-\frac\pi 3\right)\right) \id\theta \\
    &= \frac2\pi \int_{\pi/3}^{4\pi/3} \left\{\log2\left(\sin\frac\theta2\right)
    + \log2\left(\sin\frac{\theta+\frac\pi3}2\right)\right\} \id\theta \\
    &= \frac2\pi \int_{\pi/3}^{4\pi/3} \log2\left(\sin\frac\theta2\right) \id\theta
    + \frac2\pi \int_{2\pi/3}^{5\pi/3}\log2\left(\sin\frac\theta2\right) \id\theta \\
    &= \frac4\pi \mcl{2} - \frac4\pi \Cl{2}{\frac{4\pi}{3}}
    = \frac{20}{3\pi} \mcl{2},
  \end{align*}
  since $\Cl{2}{\frac{4\pi}{3}} = -\frac23 \mcl{2}$ and so we are done.
\end{proof}

When $c=1$ the cosine in the surd disappears, and we obtain $\mu_1=0$, which is
trivial as in this case the polynomial factorizes as $(1+x)^2(1+y+y^2)$. For
$c=-1$ we are able, with some care, to directly integrate \eqref{eq:mc}  and so
to obtain an apparently new Mahler measure:

\begin{theorem}We have
  \begin{align}
    \mu_{-1} &= \mu\left((x+1)^2(y^2+y+1)-2xy\right) \\
    &= \frac{1}{\pi}\,\left\{ \frac12 B\left( \frac14, \frac14 \right)
    \pFq32{\frac14,\frac14,1}{\frac34,\frac54}{\frac14}
    - \frac16 B\left( \frac34, \frac34 \right)
    \pFq32{\frac34,\frac34,1}{\frac54,\frac74}{\frac14}\right\}\nonumber.
  \end{align}
\end{theorem}

We observe that an alternative form of $\mu_{-1}$ is given by
\[ \mu_{-1} = \mu\left(\left(x+1/x+2\sqrt{1/x}\right)(y+1/y+1)-2\right). \]

\begin{remark}
  Equation \eqref{eq:lab} may be applied to other conjectured Mahler measures.
  For instance, $\mu(1+x+y+1/x+1/y)=.25133043371325\ldots$ was
  conjectured by
  Deninger \cite{finch} to evaluate in $L$-series terms as
  \begin{equation}
    \mu(1+x+y+1/x+1/y) =15 \sum_{n=1}^\infty \frac{a_n}{n^2},
  \end{equation}
  where $ \sum_{n=1}^\infty a_n q^n = \eta(q) \eta(q^3) \eta(q^5) \eta(q^{15}).$
  Here $\eta$ is the \emph{Dirichlet eta-function}:
  \begin{equation}\label{eq:eta}
    \eta(q):=q^{1/24}\, \prod_{n=1}^\infty(1-q^n)
    = q^{1/24}\,\sum_{n=-\infty}^\infty (-1)^n q^{n(3n+1)/2}.
  \end{equation}
  This has very recently been proven in \cite{rz}.

    Application of \eqref{eq:lab} shows that
  \begin{equation*}
    \mu(1+x+y+1/x+1/y) =
    \frac1\pi\, \int_0^{\pi/3}
     \log\left(\frac{1 + 2\cos\theta}2 + \sqrt{\left(\frac{1+2\cos\theta}2\right)^2 - 1}
    \right) \id\theta,
  \end{equation*}
  % \begin{equation*}
  %   \mu(1+x+y+1/x+1/y) = \frac1\pi \int_0^{\pi/3} \log\left(
  %   \frac12 + \cos\theta + \sqrt{\cos\theta + \cos^2\theta - \frac34}
  %   \right) \id\theta,
  % \end{equation*}
  but the surd remains an obstacle to a direct evaluation.
  \qed
\end{remark}

\section{Conclusion}

To recapitulate, $\mu_k(1+x+y)=W_3^{(k)}(0)$ has been evaluated in terms of
log-sine integrals for $1 \le k \le 3$. Namely,
\begin{align}\label{eq:mu2A2}
  \mu_1(1+x+y) &= \frac{3}{2\pi} \Ls{2}{\frac{2\pi}{3}}, \\
  \mu_2(1+x+y) &= \frac3\pi \Ls{3}{\frac{2\pi}{3}} + \frac{\pi^2}{4}, \\
  \mu_3(1+x+y) &\stackrel{?[1]}{=} \frac{6}{\pi}\Ls{4}{\frac{2\pi}{3}} - \frac9\pi\mcl{4}
  - \frac\pi4\mcl{2} - \frac{13}{2}\zeta(3).
\end{align}
Hence it is reasonable to ask whether $\mu_4(1+x+y)$ and higher Mahler measures
have evaluations in similar terms.

\begin{example}[Evaluation of $\mu_4(1+x+y)$]\label{ex:mu4}
  This question is taken up in \cite{logsin4} where it is found that
  \begin{align}\label{eq:mu4}
    \pi \mu_4(1+x+y) &\stackrel{?[5]}{=} 12 \Ls{5}{\frac{2\pi}{3}} -
    \frac{49}{3} \Ls{5}{\frac\pi3} + 81\Gl{4,1}{\frac{2\pi}{3}} \\
  &\quad+ 3\pi^2 \Gl{2,1}{\frac{2\pi}{3}}+ 2\zeta(3)\Cl{2}{\frac\pi3} +
   \pi\Cl{2}{\frac\pi3}^2 - \frac{29}{90}\pi^5. \nonumber
  \end{align}
  in terms of generalized Glaisher and Clausen values.
  \qede
\end{example}

We close with numerical values for these quantities.

\begin{example}%[Numerical values of $\mu_n(1+x+y)$]
  \label{ex:num}
  By computing higher-order finite differences in the right-hand side of
  \eqref{eq:hyper} we have obtained values for $\mu_n(1+x+y)$ to several
  thousand digits.  To confirm these values we have evaluated the
  double-integral \eqref{eq:mu-trigintegral} to about $250$ digits for all
  $n\le8$. These are the results for $\mu_k:=\mu_k(1+x+y)$ to fifty digits:
  \begin{align}
    \mu_2 &=  0.41929927830117445534618570174886146566170299117521,\\
    \mu_3 &=  0.13072798584098927059592540295887788768895327503289,\\
    \mu_4 &=  0.52153569858138778267996782141801173128244973155094,\\
    \mu_5 &= -0.46811264825699083401802243892432823881642492433794.
  \end{align}
  These values will allow a reader to confirm many of our results numerically.
  \qede
\end{example}

\paragraph{Acknowledgements}
We thank David Bailey for his assistance with the two-dimensional quadratures
in Example \ref{ex:num}.  Thanks are due to Yasuo Ohno and Yoshitaka Sasaki for
introducing us to the relevant papers during their recent visit to CARMA.

\bibliography{logsin-refs}
\bibliographystyle{abbrv}

\end{document}